\documentclass[12pt]{amsart}

\usepackage{amsfonts}
\usepackage{amssymb}
\usepackage{amsmath}
\usepackage{esint}
\usepackage{wasysym}
\usepackage{amsthm}
\usepackage{times} 
\usepackage{graphicx}
\usepackage{enumerate}
\usepackage{color}

\usepackage{enumerate}
\usepackage{esint}
\usepackage{color}

\newtheorem{teo}{Theorem}[section]

\newtheorem{lema}[teo]{Lemma}
\newtheorem{prop}[teo]{Proposition}

\theoremstyle{definition}
\newtheorem*{xrem}{Remark}

\numberwithin{equation}{section}

\def\real{\mathbb{R}}

\def\e{\mbox{\normalfont e}}

\let\oldmarginpar\marginpar
\renewcommand\marginpar[1]{\-\oldmarginpar[\raggedleft\footnotesize #1]
{\raggedright\footnotesize #1}}

\makeatletter\renewcommand\theenumi{\@roman\c@enumi}\makeatother

\begin{document}

\title[Maximal functions associated to radial measures.]{Bounds for maximal functions associated to rotational invariant measures in high dimensions.}

\author[A. Criado]{Alberto Criado}
\address{Alberto Criado\\ Departamento de Matem\'aticas, Universidad Aut\'onoma de Madrid, 28049 Madrid, Spain}
\email{alberto.criado@uam.es}
\thanks{A.C. partially supported by MEC grant FPU-AP20050543 and DGU grants MTM2007-60952 and MTM2010-16518.}

\author[P. Sj\"ogren]{Peter Sj\"ogren}
\address{Peter Sj\"ogren\\ Mathematical Sciences, University of Gothenburg and Mathematical Sciences, Chalmers. SE-412 96 Gothenburg, Sweden.}
\email{peters@chalmers.se}

\date{}

\begin{abstract}
In recent articles (\cite{C}, \cite{AldazPerezLazaro}) it was proved that when $\mu$ is a finite, radial measure in $\real^n$ with a bounded, radially decreasing density, the $L^p(\mu)$ norm of the associated maximal operator $M_\mu$ grows to infinity with the dimension for a small range of values of $p$ near 1. We prove that when $\mu$ is Lebesgue measure restricted to the unit ball and $p<2$, the $L^p$ operator norms of the maximal operator are unbounded in dimension, even when the action is restricted to radially decreasing functions. In spite of this, this maximal operator admits dimension-free $L^p$ bounds for every $p>2$, when restricted to radially decreasing functions. On the other hand, when $\mu$ is the Gaussian measure, the $L^p$ operator norms of the maximal operator grow to infinity with the dimension for any finite $p> 1$, even in the subspace of radially decreasing functions.
\end{abstract}


\maketitle

\section{Introduction}

\bigskip

We denote by $B(x,R)$ the ball centred at $x$ with radius $R$ with respect to a given norm on $\real^n$. For any locally integrable function $g\in L^1_{\mbox{\tiny{loc}}}(\real^n)$, we can define the associated maximal function as
$$
Mg(x)=\sup_{R>0}\frac1{|B(x,R)|}\int_{B(x,R)}|g(y)|\,dy,
$$
where by $|A|$ we denote the Lebesgue measure of the set $A$. It is a well-known fact that $M$ satisfies the inequalities
\begin{equation}\label{acotacion.l.1.debil}
|\{y\in\real^n:Mf(y)>\lambda\}|\leq \frac{c_{1,n}}\lambda \|f\|_{L^1(\real^n)},
\end{equation}
\begin{equation}\label{acotacion.l.p}
\|Mf\|_{L^p(\real^n)}\leq C_{p,n}\|f\|_{L^p(\real^n)}, \mbox{ for } 1<p<\infty.
\end{equation}
Here the subscript $n$ indicates the (possible) dimension dependence of the constant.

\bigskip

With the aim of constructing a reasonable harmonic analysis over infinite-dimensional spaces, it has been a matter of interest to determine whether bounds for maximal functions in infinite dimension can be obtained as a limit of finite-dimensional bounds. This has led to a study of the behaviour of the $L^p$ bounds for maximal functions on $\real^n$ for large $n$.

\bigskip

It was E.M. Stein who first realised (see \cite{Stein1} and \cite{SteinStromberg}) that the maximal function associated with Euclidean balls admits an $L^p$ bound independent of the dimension for every $p>1$. After that J. Bourgain \cite{Bourgain1}, \cite{Bourgain2}, \cite{Bourgain3} and A. Carbery \cite{Carbery} showed that maximal functions associated with balls resulting from arbitrary norms also have $L^p$ bounds independent of the dimension for every $p>3/2$. This was further improved by D. M\"uller \cite{Muller}, who proved that for balls given by the $l^q$ norms in $\real^n$ with $1\leq q<\infty$, the associated maximal functions admit $L^p$ inequalities with constants that can be taken independent of the dimension for each $p>1$ (see also \cite{Bourgain3}).

\bigskip

As for the constants in the weak $L^1$ inequalities, E.M. Stein and J.O. Str\"omberg \cite{SteinStromberg} showed that if $c_{1,n}$ is the smallest constant satisfying (\ref{acotacion.l.1.debil}), then $c_{1,n}=\mathcal O(n\log n)$ as $n\rightarrow\infty$ when maximal functions associated to arbitrary convex bodies are considered. For the special case of the maximal function associated to Euclidean balls, it was also proved that $c_{1,n}=\mathcal O(n)$.

\bigskip

Using the idea of discretisation by M. de Guzm\'an \cite{Guzman}, T. Men\'arguez and F. Soria \cite{MenarguezSoria} produced a method to obtain lower bounds for $c_{1,n}$. With this method, J.M. Aldaz showed in \cite{Aldaz2} that $c_{1,n}$ tends to infinity with the dimension in the case of maximal functions associated to cubes. An explicit lower bound of this growth was given in \cite{Aubrun} by G. Aubrun. In a recent work \cite{NaorTao}, A. Naor and T. Tao extended the $n \log n$ result of \cite{SteinStromberg} to the context of Ahlfors-David $n$-regular metric measure spaces. They also showed that this bound for the constant is sharp by constructing a space for which the weak $L^1$ norm of the maximal function grows like $n \log n$.

\bigskip

When restricting the action to radial functions, maximal functions associated to Euclidean balls have weak $L^1$ bounds with a constant independent of the dimension as was shown by T. Men\'arguez and F. Soria \cite{MenarguezSoria2}.

\bigskip

We can also define maximal functions when the underlying measure is not that of Lebesgue. Given a Radon measure $\mu$ in $\real^n$, the associated maximal function is defined as
$$
M_\mu g(x):=\sup_{\begin{array}{c}\ \\[-6mm] \scriptstyle{R>0}\\[-2mm]\scriptstyle{\mu(B(x,R))>0}\end{array}} \frac1{\mu(B(x,R))}\int_{B(x,R)}|g(y)|\,d\mu(y).
$$
These maximal functions also satisfy strong type $L^p(\mu)$ estimates for $1<p<\infty$ and a weak type inequality for $p=1$, namely
$$
\mu\left(\{y\in\real^n : M_\mu f(y)>\lambda\}\right)\leq\frac{c}\lambda\,\|f\|_{L^1(\mu)},
$$
\begin{equation}\label{strong.bound}
\|M_\mu f\|_{L^p(\mu)}\leq C \|f\|_{L^p(\mu)}, \mbox{ for } 1<p<\infty,
\end{equation}
as well as an $L^\infty(\mu)$ bound with constant 1. Denoting by $c_{\mu,1}$ and $C_{\mu,p}$, respectively, the best constants in the previous inequalities, the problem of finding bounds independent of the dimension can be raised also for these constants. It will be convenient to consider instead of (\ref{strong.bound}) the weak $L^p(\mu)$ bounds
\begin{equation}\label{weak.bound}
\mu\left\{(y:M_\mu g(y)>\lambda)\right\}^{1/p} \leq \frac{c}\lambda \|g\|_{L^p(\mu)},\mbox{\quad}\lambda>0.
\end{equation}
The best constant in this inequality, $c_{\mu,p}$, satisfies
\begin{equation}\label{cota.constante}
C_{\mu,p} \geq c_{\mu,p} \geq 
\frac{\lambda\ \mu(\{y\in\real^n:M_\mu g(y)\geq \lambda\})^{1/p}}{\|g\|_{L^p(\mu)}},
\end{equation}
for all $\lambda>0$ and all nonzero $g\in L^p(\mu)$. We will bound $C_{\mu,p}$ from below by means of these two inequalities. Although $C_{\mu,p}$ might be significantly larger than $c_{\mu,p}$, they cannot have a very different behaviour with respect to the dimension. Indeed, if $c_{\mu,p}$ is bounded uniformly in the dimension, then by real interpolation $C_{\mu,q}$ is also bounded with respect to the dimension for all $q>p$.

\bigskip

From now on, we shall concentrate on maximal functions associated with radial measures and Euclidean balls. Let us recall some previous results. The proof of the above bounds by means of the Besicovitch covering lemma gives that $c_{\mu,p}$ and $C_{\mu,p}$ grow at most exponentially with the dimension. If $\mu$ has a radially increasing density, the method of proof in \cite{MenarguezSoria2} applies and we obtain a dimension-free weak $L^1(\mu)$ bound for radial functions (see \cite{AdrianInfante}). In the case of a finite and radially decreasing measure, it was proved by J.M. Aldaz \cite{Aldaz2} that the best constant in the weak type $L^1(\mu)$ bound grows exponentially to infinity. In \cite{C} the first author proved that the best constants $c_{\mu,p}$ in the weak $L^p(\mu)$ inequalities also grow exponentially to infinity with the dimension for values of $p$ in a small range above one ($1\leq p<1.0048$), even when restricting the action to radial functions. A slightly better result (for $1\leq p<1.0378$) was obtained independently in \cite{AldazPerezLazaro}. In both cases, it was seen that the method used cannot give unboundedness in dimension for significantly larger values of $p$.

\bigskip

In this paper we further study this problem and give complete answers in two relevant cases. We show that if $\nu_n$ is Lebesgue measure restricted to the unit ball of $\real^n$, the best constants $c_{\nu_n,p}$ in the weak $L^p(\nu_n)$ bounds for the associated maximal operators tend to infinity with the dimension for $1\leq p<2$. However, this maximal operator admits a dimension-free $L^2(\nu_n)$ bound of restricted weak type on radially decreasing functions, which implies a similar strong-type bound in $L^p(\nu_n)$ for $p>2$. We will also study the relation between this maximal function and the Hardy operator. On the other hand, we will prove that the maximal function associated to the Gaussian measure does not admit weak $L^p$ bounds independent of the dimension for any $p\in [1,\infty)$, not even when the action is restricted to radially decreasing functions.

\bigskip

Thus, different families of measures may have very different behaviour, and that is why general results like the ones in \cite{C} and \cite{AldazPerezLazaro} work only in a small range of values of $p$.

\bigskip

The statements of the results described above are given in the next section. Section 3 contains some notations and two lemmas that will be used in the proofs. Section 4 is devoted to the proofs of the results related to Lebesgue measure restricted to the unit ball. Finally, Section 5 deals with the proofs of the results related to the Gaussian measure.

\bigskip

\section{Statements of results}

\bigskip

Let $\nu_n$ be the measure on $\real^n$ whose density function is the characteristic function of the unit ball. We will prove that the best constants in the weak $L^p(\nu_n)$ inequalities grow exponentially to infinity with the dimension if $1\leq p<2$.

\bigskip

\begin{teo}\label{no.acotacion.hasta.l2}
If $1\leq p<2$, there exist constants $a_p>1$ and $c_p>0$ such that for all $n$
\[
c_{\nu_n,p}>c_p\,a_p^n.
\]
\end{teo}

\bigskip

To obtain this result, we will look at the action of the maximal operators on characteristic functions of balls centred at the origin. We will denote by $B_r$ the ball with radius $r$ centred at the origin, and let $\chi_r=\chi_{B_r}$ be its characteristic function. However, for $p \geq 2$ these functions cannot be used to find a counterexample; indeed, we will show that the action of $M_{\nu_n}$ on them has weak $L^{p}(\nu_n)$ bounds independent of the dimension:

\bigskip

\begin{prop}\label{lp.caracteristicas.bolas}
Let $p\geq 2$ and $r>0$. Then
\[
\|M_{\nu_n}\chi_r\|_{L^{p,\infty}(\nu_n)}^\ast\leq 2^{2/p}\|\chi_r\|_{L^p(\nu_n)}.
\]
\end{prop}

\bigskip

For Lorentz spaces $L^{p,q}$ and their norms and quasinorms, see Section \ref{section.notations}.

\bigskip

As a consequence of this proposition we obtain a dimension-free restricted weak-type $L^p(\nu_n)$ inequality for radial decreasing functions.

\bigskip

\begin{teo}\label{debil.restringida.lp}
Let $g$ be a radial, decreasing function in $\real^n$. Then for $p\geq2$
\[
\|M_{\nu_n}g\|_{L^{p,\infty}(\nu_n)}^\ast\leq \frac p{p-1} 2^{2/p} \|g\|_{L^{p,1}(\nu_n)}^\ast.
\]
\end{teo}

\bigskip

One can check that the proof of the Marcinkiewicz theorem for Lorentz spaces (see \cite{SteinWeiss}, page 197) is also valid when restricting the action of the operator to radially decreasing functions. This allows us to interpolate between the case $p=2$ of Theorem \ref{debil.restringida.lp} and the $L^\infty(\nu_n)$ inequality to obtain

\bigskip

\begin{teo}\label{lp.con.marcinkiewicz}Let $g$ be a radially decreasing function. One has for $p>2$
\begin{eqnarray*}
\|M_{\nu_n}g\|_{L^{p,\infty}(\nu_n)}^\ast &\leq& 2^{1/p}\,\frac{5p-2}{p-2}\,\|g\|_{L^p(\nu_n)},\\
\|M_{\nu_n}g\|_{L^p(\nu_n)}&\leq&2^{1/p}\,\frac{5p-2}{p-2}\,\|g\|_{L^p(\nu_n)}.
\end{eqnarray*}
\end{teo}

\bigskip

However, as we will show, there is a more direct way to obtain weak-type $L^p$ bounds. For this, we will control $M_{\nu_n}$ by a modified Hardy operator. Defining the Hardy operator for a locally integrable function $g$ in $\real^n$ as
$$
\mathcal Ag(x)=\frac1{|B_{|x|}|}\int_{B_{|x|}}|g(y)|\,dy,
$$
we have the following estimate.

\bigskip

\begin{prop}\label{operador.hardy}
Given a radially decreasing function $g$, one has for $p>2$
$$
M_{\nu_n}g(x)\leq \frac{p+2}{p-2}\left(\mathcal A g^p(x)\right)^{1/p},
$$
for each $x\in\real^n$.
\end{prop}

\bigskip

This is useful, because we can bound the operator $g\mapsto(\mathcal Ag^{p})^{1/p}$ as follows.

\bigskip

\begin{prop}\label{acotacion.operador.hardy} If $g$ is a radially decreasing function in $\real^n$ and $p\geq 1$, then
\[
\|(\mathcal A g^p)^{1/p}\|_{L^{p,\infty}(\nu_n)}^{\ast}\leq \|g\|_{L^p(\nu_n)}.
\]
\end{prop}

\bigskip

As an immediate consequence of Propositions \ref{operador.hardy} and \ref{acotacion.operador.hardy}, we have the following weak $L^p(\nu_n)$ bound for $M_{\nu_n}$, sharper than the one in Theorem \ref{lp.con.marcinkiewicz}.

\bigskip

\begin{teo}\label{dimension.free.lp.bounds}
Let $g:\real^n\,\longrightarrow\,\real$ be a radial, decreasing function. For each $p\in(2,\infty]$ one has
$$
\|M_{\nu_n}g\|_{L^{p,\infty}(\nu_n)}^{\ast}\leq \frac{p+2}{p-2}\|g\|_{L^p(\nu_n)}.
$$
\end{teo}

\bigskip

The other case that we will study here is the Gaussian measure $d\gamma_n(x)=\e^{-\pi|x|^2}\,dx$. In this case we will prove that the associated maximal function does not admit dimension-free $L^p(\gamma_n)$ bounds for any $1\leq p <\infty$:

\bigskip

\begin{teo}\label{teorema.gaussiana}
There exist absolute constants $a>1$ and $c>0$ such that for every $p$ in the range $1\leq p <\infty$,
$$
c_{\gamma_n,p} \geq c\, a^{n/p}.
$$
\end{teo}

\bigskip

This result can be extended to the case where the density is $f_\alpha(|x|)=\e^{-|x|^\alpha}$, with $\alpha>0$.

\bigskip

\begin{teo}\label{teorema.exponencial}
Let $\gamma_{\alpha,n}$ be the measure given for $\alpha>0$ by $d\gamma_{\alpha,n}(x)=\e^{-|x|^\alpha}\,dx$. There exist constants $a=a(\alpha)>1$ and $c_\alpha>0$ such that the corresponding weak $L^p(\gamma_{\alpha,n})$ constant satisfies
$$
c_{\gamma_n^\alpha,p}\geq c_\alpha\, a_\alpha^{n/p}.
$$
for every $n$ and $1\leq p<\infty$.
\end{teo}

\bigskip

\section{Notations and technical lemmas.}\label{section.notations}

\bigskip

The following lemma will be the starting-point in the proofs of Proposition \ref{lp.caracteristicas.bolas} and Theorems \ref{teorema.gaussiana} and \ref{teorema.exponencial}.

\bigskip

\begin{lema}\label{prop.estimacion.constante} Let $\mu$ be a rotation-invariant Radon measure in $\real^n$. Then for each $x\in\real^n$ and $r,R>0$ such that $\mu(B(x,R))>0$, we have that
\begin{equation}\label{estimacion.constante}
c_{\mu,p}\geq M_\mu \chi_r(x)\left(\frac{\mu(B_{|x|})}{\mu(B_r)}\right)^{1/p}\geq \frac{\mu(B(x,R)\cap B_r)}{\mu(B(x,R))}\left(\frac{\mu(B_{|x|})}{\mu(B_r)}\right)^{1/p}.
\end{equation}
\end{lema}

\bigskip

In the proof of this lemma we will use the following result, valid for a general measure $\mu$.

\bigskip

\begin{lema}\label{maximal.es.decreciente} Let $\mu$ be a Radon measure on $\real^n$. Then the maximal function $M_\mu \chi_r$ is decreasing on each ray from the origin. That is, for any $x\in\real^n$ and $y=\alpha x$ with $0<\alpha<1$ we have that
\begin{equation}
M_\mu \chi_r(x)\leq M_\mu \chi_r(y).
\end{equation}
\end{lema}

\bigskip

\begin{proof}[Proof that Lemma \ref{maximal.es.decreciente} implies Lemma \ref{prop.estimacion.constante}] The first inequality in (\ref{estimacion.constante}) follows if we let $g=\chi_r$ and $\lambda=M_\mu \chi_r(x)$ in (\ref{cota.constante}), since Lemma \ref{maximal.es.decreciente} implies that $M_\mu \chi_r(y)\geq M_\mu \chi_r(x)$ for any $y\in B_{|x|}$. The second inequality is easy.
\end{proof}

\bigskip

\begin{proof}[Proof of Lemma \ref{maximal.es.decreciente}]
First we discard the trivial case when $y \in B_r$, since then $M_\mu \chi_r(y) = 1$ and we always have $M_\mu \chi_r(x) \leq 1$. Assume that $y$ and (consequently) $x$ are not in $B_r$. It would be enough to show that for each $R>0$ we can find a $T>0$ such that
\[
\frac {\mu(B(x,R)\cap B_r)}{\mu(B(x,R))} \leq \frac{\mu(B(y,T)\cap B_r)}{\mu(B(y,T))}.
\]
Take $T$ such that $\partial B(x,R) \cap \partial B_r = \partial B(y,T) \cap \partial B_r$. We call
\[
\begin{array}{rclcrcl}
A &=& \mu(B(y,T)\setminus B(x,R)), &\quad& B &=& \mu(B(x,R)\cap B_r),\\
C &=& \mu(B(y,T)\setminus B_r), &\quad& D &=& \mu(B(x,R)\setminus B(y,T)).
\end{array}
\]
Now it is clear that
\[
\frac {\mu(B(x,R)\cap B_r)}{\mu(B(x,R))}=\frac B{B+C+D} \leq \frac B{B+C} \leq \frac{A+B}{A+B+C} = \frac{\mu(B(y,T)\cap B_r)}{\mu(B(y,T))}
\]
\end{proof}

\bigskip

\begin{xrem}
In order to obtain lower bounds for $c_{\mu,p}$ using (\ref{estimacion.constante}), there is no point in considering the case $|x|<r$, since it will never lead to a lower bound greater than 1.
\end{xrem}

\bigskip

We now introduce solid spherical caps. Given a ball $B_\rho$ and a vector $y\neq 0$ in $B_\rho$, consider the hyperplane $y+y^\bot$, which divides the ball into two closed sets. We focus on the one of these sets which does not contain the origin. Its diameter is $2L=2\sqrt{\rho^2-|y|^2}$.  We denote this set by $A_{\rho}(L)$, and any set congruent with it will be called a solid spherical cap. The height of this cap is given by the function
\begin{equation}\label{height.of.cap}
h(\rho,L)=\rho-\sqrt{\rho^2-L^2}.
\end{equation}
Denoting by $\omega_{k-1}$ the area of the unit sphere in $\real^k$, we have for the Lebesgue measure of $A_\rho(L)$
\[
|A_{\rho}(L)| = \int_{|y|}^\rho \frac{\omega_{n-2}}{n-1} (\rho^2-s^2)^\frac{n-1}2\,ds = \frac{\omega_{n-2}}{n-1} \int_0^L\frac{t^n}{\sqrt{\rho^2-t^2}}\,dt,
\]
where the last equality comes from the change of variables $t=\sqrt{\rho^2-s^2}$. Since $\sqrt{\rho^2-L^2}<\sqrt{\rho^2-t^2}< \rho$, we obtain that
\begin{equation}\label{medida.casquete}
\frac{\omega_{n-2}}{n^2-1}L^{n+1}\frac1{\rho} \leq |A_{\rho}(L)|\leq \frac{\omega_{n-2}}{n^2-1}L^{n+1}\frac1{\sqrt{\rho^2-L^2}}.
\end{equation}

\bigskip

We finish this section by briefly stating the definitions and some properties of Lorentz spaces that will be used. Let $\mu$ be a Radon measure in $\real^n$. Given a measurable function $f$, we denote by $f^\ast$ its non-increasing rearrangement with respect to $\mu$. Let $1\leq p<\infty$. The quasinorm of $f$ in the Lorentz space $L^{p,q}(\mu)$ with $1\leq q<\infty$ is defined by
\[
\|f\|_{L^{p,q}(\mu)}^\ast = \left(\frac qp\int_0^\infty [s^{1/p}f^\ast(s)]^q\,\frac{ds}s\right)^{1/q},
\]
and for $q=\infty$ by
\[
\|f\|_{L^{p,\infty}(\mu)}^\ast = \sup_{s>0} s^{1/p} f^\ast(s)=\sup_{\lambda>0}\lambda\,\mu(\{|f|>\lambda\})^{1/p},
\]
with the usual agreement that $\|f\|_{L^{\infty,\infty}(\mu)}^\ast = \|f^\ast\|_{L^{\infty}(\real)} = \|f\|_{L^{\infty}(\mu)}$. In most of the cases this is not a norm, since the triangle inequality may fail. However, the spaces $L^{p,q}(\mu)$ admit a norm denoted $\|\cdot\|_{L^{p,q}(\mu)}$ for $1<p\leq\infty$ and $1\leq q\leq\infty$.
%
As quasinorms, $\|\cdot\|_{L^{p,q}(\mu)}^\ast$ and $\|\cdot\|_{L^{p,q}(\mu)}$ are equivalent in the sense that
\[
\|f\|_{L^{p,q}(\mu)} \leq \|f\|_{L^{p,q}(\mu)}^\ast \leq \frac p{p-1}\|f\|_{L^{p,q}(\mu)},
\]
for any $1<p<\infty$ and $1\leq q\leq \infty$.

\bigskip

For more details, see Chapter V, \S 3 of \cite{SteinWeiss}.

\bigskip

\section{Lebesgue measure restricted to the unit ball.}\label{lebesgue.measure}

\bigskip

This section is mainly devoted to the proofs of Theorem \ref{no.acotacion.hasta.l2} and Proposition \ref{lp.caracteristicas.bolas}.

\bigskip

The following lemma explains why in both proofs it is enough to concentrate on the situation when $|x|=1$.

\bigskip

\begin{lema}\label{frontera.bola.unidad}
For any $r<1$ and each $x\in B_1$ we have that
$$
M_{\nu_n}\chi_r(x)\leq M_{\nu_n}\chi_{r/|x|}\left(\frac x{|x|}\right).
$$
\end{lema}

\bigskip

\begin{proof}
For any $R>0$
\[
\frac{|B(x,R)\cap B_r|} {|B(x,R)\cap B_1|} \leq \frac{|B(x,R)\cap B_r|}{|B(x,R)\cap B_{|x|}|} = \frac{|B\left(\frac x{|x|},\frac R{|x|}\right)\cap B_{\frac r{|x|}}|}{|B\left(\frac x{|x|},\frac R{|x|}\right)\cap B_1|},
\]
and, as $R>0$ is arbitrary, Lemma \ref{frontera.bola.unidad} is proved.
\end{proof}

\bigskip

\begin{proof}[Proof of Theorem \ref{no.acotacion.hasta.l2}]
In view of Lemma \ref{frontera.bola.unidad}, we take a unit vector $x$. Lemma \ref{prop.estimacion.constante} asserts that
\begin{equation}\label{estimacion.constante.bola.unidad}
c_{\nu_n,p}\geq \frac{|B(x,R)\cap B_r|}{|B(x,R)\cap B_1|}\left(\frac1r\right)^\frac np,
\end{equation}
for $R,r>0$. With $r<1$ and $1-r<R<r$, we shall choose $r$ close to 1 and $R$ small.

\bigskip

We split the intersection of two balls into two solid spherical caps and conclude from (\ref{medida.casquete}) that
\begin{equation}
|B(x,R)\cap B_1|=|A_R(L)|+|A_1(L)| \leq \frac{\omega_{n-2}}{n^2-1}L^{n+1}\left(\frac1{\sqrt{R^2-L^2}}+\frac1{\sqrt{1-L^2}}\right),
\end{equation}
where $L=\sqrt{R^2-R^4/4}$.
In the same fashion
\begin{equation}
|B(x,R)\cap B_r|=|A_R(\ell)|+|A_r(\ell)| \geq \frac{\omega_{n-2}}{n^2-1}\ell^{n+1}\left(\frac1R+\frac1r\right),
\end{equation}
with $\ell=\sqrt{R^2-(R^2-r^2+1)^2/4}=\sqrt{r^2-(r^2-R^2+1)^2/4}$.
Putting these two estimates together, we obtain
\[
\frac{|B(x,R)\cap B_r|}{|B(x,R)\cap B_1|} \geq \beta \left(\frac \ell L\right)^{n+1},
\]
where $\beta=\beta(r,R)>0$ is independent of $n$.
We set
\[
\phi_r(R) = \left(\frac \ell L\right)^2 = \frac{4R^2-(R^2-r^2+1)^2}{4R^2-R^4}.
\]

Inequality (\ref{estimacion.constante.bola.unidad}) implies now
\[
c_{\nu_n,p} \geq \beta\,\phi_r(R)^{(n+1)/2}\,r^{-n/p}.
\]
It is enough to show that $r$ and $R$ can be chosen so that $\phi_r(R)^{1/2}\,r^{-1/p}>1$. This is equivalent to
\[
p<\frac{2\log r}{\log \phi_r(R)}.
\]
By setting $t=1-r^2$, we obtain
\[
\phi_r(R)=1-\frac{2R^2t+t^2}{4R^2-R^4},
\]
and with the choice $R=t^{1/4}$ one has
\[
\phi_r(t^{1/4})=1-\frac{2t-t^{3/2}}{4-t^{1/2}}=1-\frac t2 + o(t),
\]
as $t\longrightarrow 0$. Now Theorem \ref{no.acotacion.hasta.l2} is proved, because
\[
\lim_{r\rightarrow 1^-} \frac{2\log r}{\log \phi_r(t^{1/4})}= \lim_{t\rightarrow 0^+} \frac{\log (1-t)}{\log (1- t/2 + o(t))}=2.
\]

\end{proof}

\bigskip

We will obtain Proposition \ref{lp.caracteristicas.bolas} as a consequence of the following result, whose proof will appear at the end of this section.

\bigskip

\begin{prop}\label{cociente.bolas} Given $x\neq 0$ in $\real^n$, \ $0< r \leq 1$ and $R>0$, one has
\begin{equation}\label{cota.cociente.bolas}
\frac{|B(x,R)\cap B_r|}{|B(x,R)\cap B_1|}\leq 2 \left(\frac{|B_r|}{|B_{|x|}|}\right)^\frac12.
\end{equation}
Equivalently, for each $x\in \real^n$ and $0<r \leq 1$, one has
\[
M_{\nu_n}\chi_r(x)\leq 2 \left(\frac{r}{|x|}\right)^\frac n2.
\]
\end{prop}

\bigskip

\begin{xrem}
It may be the case that the best constant in (\ref{cota.cociente.bolas}) is 1 rather than 2, but we have made not effort to compute its exact value.
\end{xrem}

\bigskip

\begin{proof}[Proof of Proposition \ref{lp.caracteristicas.bolas}]
The result is trivial when $r\geq 1$. We just have to show that given $0<r<1$, for each $x\in \real^n$,
\[
M_{\nu_n}\chi_r(x)\ \nu_n\!\left(\{y\in\real^n:M_{\nu_n}\chi_r(y) > M_{\nu_n}\chi_r(x) \}\right)^{1/p} \leq 2^{1/p} |B_r|^{1/p}.
\]
Since $M_{\nu_n}\chi_r(x)\leq 1$, by Proposition \ref{cociente.bolas}
\[
M_{\nu_n}\chi_r(x) \leq M_{\nu_n}\chi_r(x)^{2/p} \leq 2^{2/p}\left(\frac r{|x|}\right)^{n/p}.
\]
In view of Lemma \ref{maximal.es.decreciente}
\[
\{y\in\real^n:M_{\nu_n}\chi_r(y) > M_{\nu_n}\chi_r(x) \} \subset B_{|x|}.
\]
Hence,
\begin{eqnarray*}
M_{\nu_n}\chi_r(x)\ \nu_n\!\left(\{y\in\real^n:M_{\nu_n}\chi_r(y) > M_{\nu_n}\chi_r(x) \}\right)^{1/p} &\leq& 2^{2/p}\left(\frac r{|x|}\right)^{n/p}|B_{|x|}|^{1/p}\\&=& 2^{2/p} |B_r|^{1/p}.
\end{eqnarray*}
\end{proof}

\bigskip

\begin{proof} [Proof of Theorem \ref{debil.restringida.lp}]
By a density argument it is enough to prove the result for a simple function of the form $g=\sum_{i=1}^N c_i \chi_{B_i}$, where $B_1\supset\cdots \supset B_i \supset\cdots \supset B_N$ are balls centred at the origin and $c_i,\ i=1,\ldots,N$ are positive real numbers. Since $M_{\nu_n}$ is a sublinear operator, we have for such a function $g$
\begin{eqnarray*}
\|M_{\nu_n}g\|_{L^{p,\infty}(\nu_n)}^\ast &\leq& \|M_{\nu_n}g\|_{L^{p,\infty}(\nu_n)} \leq \sum_{i=1}^N c_i \|M_{\nu_n}\chi_{B_i}\|_{L^{p,\infty}(\nu_n)}\\ &\leq& \frac p{p-1} \sum_{i=1}^N c_i \|M_{\nu_n}\chi_{B_i}\|_{L^{p,\infty}(\nu_n)}^\ast.
\end{eqnarray*}
By Proposition \ref{lp.caracteristicas.bolas}
\begin{eqnarray*}
\sum_{i=1}^N c_i \|M_{\nu_n}\chi_{B_i}\|_{L^{p,\infty}(\nu_n)}^\ast &\leq& 2^{2/p} \sum_{i=1}^N c_i \,\|\chi_{B_i}\|_{L^{p}(\nu_n)} =
2^{2/p}\sum_{i=1}^N c_i \,\nu_n(B_i)^{1/p} \\&=& 2^{2/p}\|g\|_{L^{p,1}(\nu_n)}^\ast.
\end{eqnarray*}
\end{proof}

\bigskip

\begin{proof}[Proof of Proposition \ref{operador.hardy}]
By homogeneity we can assume that the radially decreasing  function $g$ satisfies $\mathcal A g^p\,(x) = 1$. Given $t>0$, the level set $\{y:g(y)>t\}$ is the ball $B_{r(t)}$ for a certain $r(t) > 0$. Thus
\begin{eqnarray*}
M_{\nu_n}g(x)&=&\sup_{R>0}\frac1{|B(x,R)\cap B_1|}\int_{B(x,R)\cap B_1}g(y)\,dy\\&=&\frac1{|B(x,R)\cap B_1|} \int_0^\infty|\{y\in B(x,R)\cap B_1:g(y)>t\}|\,dt\\&=& \frac1{|B(x,R)\cap B_1|} \int_0^\infty|B(x,R)\cap B_{r(t)}\cap B_1|\,dt\\&=& \int_0^1\frac{|B(x,R)\cap B_{r(t)}\cap B_1|}{|B(x,R)\cap B_1|}\,dt+ \int_1^\infty \frac{|B(x,R)\cap B_{r(t)}\cap B_1|}{|B(x,R)\cap B_1|}\,dt.
\end{eqnarray*}
The first term on the last line is clearly bounded by 1. For the second one, we can use Proposition \ref{cociente.bolas} to get
$$
M_{\nu_n}g(x)\leq\ 1+ 2\int_1^\infty\frac{|B_{r(t)}|^{1/2}}{|B_{|x|}|^{1/2}}\,dt \leq 1+\frac{2}{|B_{|x|}|^{1/2}} \int_{1}^\infty|\{y:g(y)>t\}|^{1/2}\,dt.
$$
The hypothesis $\mathcal A g^p(x)=1$ implies $g(x)\leq 1$, so for $g(y)>t>1$ it is necessary that $y\in B_{|x|}$. Then, by the Tchebychev inequality, the above expression is less than or equal to
\begin{eqnarray*}
1+\frac{2}{|B_{|x|}|^{1/2}} \int_1^\infty\frac1{t^{p/2}}\left(\int_{B_{|x|}}g(y)^p\,dy\right)^{1/2}\,dt &=& 1+ \frac{4}{p-2}\left(\frac1{|B_{|x|}|}\int_{B_{|x|}}g(y)^p\,dy\right)^{1/2} \\&=&1+\frac{4}{p-2}.
\end{eqnarray*}
\end{proof}

\bigskip

\begin{proof}[Proof of Proposition \ref{acotacion.operador.hardy}]
If $g$ is radially decreasing, so is $(\mathcal Ag^p)^{1/p}$, and its level sets are balls centred at the origin. So given $\lambda>0$
\[
\{y:(\mathcal Ag^p)^{1/p}(y)>\lambda\}=B_{r(\lambda)},
\]
for some $r(\lambda)>0$. Hence
\[
\left(\frac1{|B_{r(\lambda)}|}\int_{B_{r(\lambda)}}g(y)^p\,dy\right)^{1/p} \geq \lambda,
\]
which we can rearrange as
\[
\left|\{y:(\mathcal Ag^p)^{1/p}(y)>\lambda\}\right|=\left|B_{r(\lambda)}\right|\leq \frac1{\lambda^p}\int_{B_{r(\lambda)}}g(y)^p\,dy.
\]
\end{proof}

\bigskip

\begin{proof}[Proof of Proposition \ref{cociente.bolas}]
It is enough to prove the result in the case when $|x|=1$, because then by Lemma \ref{frontera.bola.unidad}
\[
M_{\nu_n}\chi_r(x)\leq M_{\nu_n}\chi_{\frac r{|x|}}(x/|x|)\leq 2 \left(\frac r{|x|}\right)^{n/2}.
\]

\bigskip

So, assuming that $|x|=1$, we want to prove that for each $R>0$ and $0<r\leq 1$
\begin{equation}\label{goal}
\frac{|B(x,R)\cap B_r|}{|B(x,R)\cap B_1|}\leq 2 r^{n/2}.
\end{equation}
The case where $R\,\leq\, 1-r$ is trivial since then $|B(x,R)\cap B_r|=0$, so from now on we assume $R>1-r$. Using the same notation as in the proof of Theorem \ref{no.acotacion.hasta.l2}, we have that
\begin{eqnarray}
|B(x,R)\cap B_r|&=&|A_R(\ell)|+|A_r(\ell)|,\label{caps1}\\
|B(x,R)\cap B_1|&=&|A_R(L)|+|A_1(L)|.\label{caps2}
\end{eqnarray}
We first prove the inequality $\ell/L\leq r^{1/2}$. The equivalent statement $\ell^2\leq L^2r$ can be rewritten as $-R^4+2R^2(1-r)-(1-r)(1+r)^2\leq 0.$
This is a second-degree polynomial in $t=R^2$, whose maximal value, assumed at $t=1-r$, is $(1-r)^2-(1-r)(1+r)^2\leq 0$ for each $r$ in $[0,1]$.
\bigskip

We divide the proof of (\ref{goal}) into three cases:\\
\noindent\textsc{Case 1:} $R\leq r$. Using (\ref{caps1}) and (\ref{caps2}) we have that
\begin{eqnarray*}
|B(x,R)\cap B_r|&\leq& 2|A_R(\ell)|,\\
|B(x,R)\cap B_1| &\geq& |A_R(L)|.
\end{eqnarray*}
Dilating by the factor $\ell/L<1$, we get $\frac \ell LA_R(L) = A_{\frac \ell LR}(\ell)$. This implies that $|A_R(\ell)|\leq |A_{\frac \ell LR}(\ell)|= (\ell/L)^n |A_R(L)|$. So we have
\[
\frac{|B(x,R)\cap B_r|}{|B(x,R)\cap B_1|} \leq \frac{2|A_R(\ell)|}{|A_R(L)|} \leq 2\left(\frac \ell L\right)^n \leq 2r^{n/2}.
\]

\medskip

\noindent\textsc{Case 2:} $r<R\leq\sqrt{1+r^2}$. In this situation (\ref{caps1}) implies
\[
|B(x,R)\cap B_r|\leq 2|A_r(\ell)|.
\]
Here we dilate $A_r(\ell)$ by the factor $r^{-1/2}$ instead: $r^{-1/2} A_r(\ell)= A_{r^{1/2}}(\ell/r^{1/2})$. We claim that a set congruent with $A_{r^{1/2}}(\ell/r^{1/2})$ is contained in $B(x,R)\cap B_1$. This would give the bound $|A_{r^{1/2}}(\ell/r^{1/2})| \leq |B(x,R)\cap B_1|$, and as a consequence
\begin{equation}\label{caso.2}
\frac{|B(x,R)\cap B_r|}{|B(x,R)\cap B_1|} \leq \frac{2\,|A_r(\ell)|}{|B(x,R)\cap B_1|} \leq 2\,\frac{r^{n/2}\,|A_{r^{1/2}}(\ell/r^{1/2})|}{|B(x,R)\cap B_1|} \leq 2\,r^{n/2}.
\end{equation}
Thus we only have to justify the claim. For this we regard $B(x,R)\cap B_1$ as the union of two solid spherical caps $\tilde A_1(L)$ and $\tilde A_R(L)$ congruent with $A_1(L)$ and $A_R(L)$, respectively. Consider the unique hyperplane parallel to the planar boundary of $\tilde A_R(L)$ whose intersection with $\tilde A_R(L)$ is a circular disc $D$ of radius $\ell/r^{1/2}$. This disc divides $\tilde A_R(L)$ into two sets. One of them, $\tilde A_R(\ell/r^{1/2})$, is congruent with $A_R(\ell/r^{1/2})$. Call $\tilde A_{r^{1/2}}(\ell/r^{1/2})$ the cap congruent with $A_{r^{1/2}}(\ell/r^{1/2})$ such that $\tilde A_R(\ell/r^{1/2}) \cap \tilde A_{r^{1/2}}(\ell/r^{1/2})=D$. To see that $\tilde A_{r^{1/2}}(\ell/r^{1/2})$ is contained in $\tilde A_1(L)\cup \tilde A_R(L)$, and thus in $B(x,R)\cap B_1$, it is enough to compare the heights of four caps and verify that
\[
h(r^{1/2},\ell/r^{1/2})\leq h(1,L)+h(R,L)-h(R,\ell/r^{1/2}).
\]
In view of the definition (\ref{height.of.cap}) of $h$, it is not difficult to see that $h(1,L)+h(R,L)=R$, and the above inequality becomes
\[
r^{1/2}-\sqrt{r-\ell^2/r}\leq \sqrt{R^2-\ell^2/r}.
\]
We can multiply by $r^{1/2}$ on both sides and use that $\ell^2=r^2-((r^2-R^2+1)/2)^2$ to get the equivalent statement
\[
r-\frac{r^2-R^2+1}2\leq \sqrt{R^2r-r^2+\left(\frac{r^2-R^2+1}2\right)^2}.
\]
Here the left-hand side is positive since $R>1-r$, and one obtains by squaring the equivalent inequality
\[
-r(1-r)^2\leq0,
\]
which holds since $0<r \leq 1$. The claim follows.

\bigskip

\noindent\textsc{Case 3:} $R>\sqrt{1+r^2}$.
In this case, the ball $B(x,R)$ contains more than half of the ball $B_r$. We have
\[
\frac{|B(x,R)\cap B_r|}{|B(x,R)\cap B_1|}\leq \frac{|B_r|}{|B(x,\sqrt{1+r^2})\cap B_1|}= \frac{2|A_r(r)|}{|B(x,\sqrt{1+r^2})\cap B_1|},
\]
and now we can use (\ref{caso.2}) in the special case where $R=\sqrt{1+r^2}$ and $\ell=r$ to get
\[
\frac{2|A_r(r)|}{|B(x,\sqrt{1+r^2})\cap B_1|}\leq 2\,r^{n/2}.
\]

\bigskip

\end{proof}

\bigskip

\section{The Gaussian measure}\label{gaussian.measure}

\bigskip

In this section we will state some properties of the Gaussian measure, and we will give the proofs of Theorems \ref{teorema.gaussiana} and \ref{teorema.exponencial}.

\bigskip

The idea of the proof of Theorem \ref{teorema.gaussiana} is the following. From Lemma \ref{prop.estimacion.constante} we know that for any $x_n\in\real^n$ and $r_n>0$
\begin{equation}\label{constante.gaussiana.bolas}
c_{\gamma_n,p} \geq M_{\gamma_n}\chi_{r_n}(x_n)
\left(\frac{\gamma_n(B_{|x_n|})}{\gamma_n(B_{r_n})}\right)^\frac1p.
\end{equation}
Since $M_{\gamma_n}\chi_{r_n}(x_n)\geq \gamma_n(B(x_n,R_n)\cap B_{r_n})/\gamma_n(B(x_n,R_n))$ for each $R_n>0$, we only need to prove the following lemma.

\bigskip

\begin{lema}\label{prueba.gaussiana}
There exist sequences $\{x_n\}, \{r_n\}$ and $\{R_n\}$ with $x_n\in \real^n$ and $r_n, R_n>0$ for $n\in\mathbb N$, such that
\begin{equation}\label{estimacion.maximal.gaussiana}
\ \frac{\gamma_n(B(x_n,R_n)\cap B_{r_n})}{\gamma_n(B(x_n,R_n))}\geq \frac c{\sqrt{n}},
\end{equation}
\begin{equation}\label{cociente.bolas.centradas.medida.gaussiana}
\frac{\gamma_n(B_{|x_n|})}{\gamma_n(B_{r_n})}\geq c\,a^n,
\end{equation}
for some absolute constants $a>1$ and $c>0$.
\end{lema}

\bigskip

Proving (\ref{cociente.bolas.centradas.medida.gaussiana}) will involve dealing with the Gaussian measure of balls centred at the origin. The measure of $B_R$ is
$$
\gamma_n(B_R)=\omega_{n-1}\int_0^R \e^{-\pi s^2}s^{n-1}\,ds.
$$
The function $G(s):=\e^{-\pi s^2}s^{n-1}$ is increasing in the interval $(0,T_n)$ and decreasing in $(T_n,\infty)$, where $T_n:=\sqrt{\frac{n-1}{2\pi}}$. So $G$ attains its maximum at the point $s=T_n$. An essential part of the mass of $\gamma_n$ is concentrated around the sphere of radius $T_n$.
\begin{lema}\label{estimaciones.bolas.centradas.medida.gaussiana}
If $\rho<T_n$ the following estimates hold
\[\e^{-\pi \rho^2}|B_\rho| \leq \gamma_n(B_\rho) \leq n \e^{-\pi \rho^2}|B_\rho|.\]
\end{lema}

\bigskip

The proofs of these facts are easy (see \cite{C} for details).

\bigskip

\begin{proof}[Proof of Lemma \ref{prueba.gaussiana}]
We start with statement (\ref{cociente.bolas.centradas.medida.gaussiana}). Let us take $x_n=T_n x_n'$ with $x_n'$ a unit vector in $\real^n$. In order to make the quotient $|x_n|/r_n$ independent of $n$, set $r_n=rT_n$, with $0<r<1$. Lemma \ref{estimaciones.bolas.centradas.medida.gaussiana} implies that
$$
\frac{\gamma_n(B_{|x_n|})}{\gamma_n(B_{r_n})} \geq \frac{\e^{-\pi T_n^2}|B_{ T_n}|}{n\,{\e^{-\pi r^2 T_n^2}}|B_{r T_n}|} = \frac{\e^\frac{1-r^2}2}n\left(\e^\frac{r^2-1}2\frac1r\right)^n.
$$
To see that the quantity in the last parenthesis is greater than 1, just apply the inequality $e^x > 1+x$ with $x=r^2-1$. This proves (\ref{cociente.bolas.centradas.medida.gaussiana}).

\bigskip

We now turn to the proof of (\ref{estimacion.maximal.gaussiana}). Take $R_n=RT_n$, with $1-r<R<1$. To calculate $\gamma_n(B(x_n,R_n))$, we will integrate over spherical caps where the density $\e^{-\pi|x|^2}$ is constant. We get
\begin{eqnarray*}
\gamma_n(B(x_n,R_n))&=&\int_{T_n-R_n}^{T_n+R_n}\Bigl|\partial B_\rho\cap B(x_n,R_n)\Bigr|_{n-1} \e^{-\pi\rho^2}\,d\rho \\&=& T_n^n\int_{1-R}^{1+R}\Bigl|\partial B_s\cap B(x_n', R)\Bigr|_{n-1} \e^{-\frac{n-1}2s^2}\,ds,
\end{eqnarray*}
where $|\cdot|_{n-1}$ denotes ($n-1$)-dimensional Hausdorff measure, and the second equality is justified by the change of variables $\rho=T_ns$. Call $\beta_s$ the angle determined by the segment that joins the origin with $x_n'$ and the one that connects the origin with any point. Then $\beta_s<\pi/2$ since $R<1$. We compute the surface measure of the spherical caps in the following way
\begin{equation}\label{surface.cap}
\Bigl|\partial B_s\cap B(x_n', R)\Bigr|_{n-1}=\int_0^{\beta_s}\omega_{n-2}(s\sin\theta)^{n-2}s\,d\theta.
\end{equation}
For the last integral we have
\[
\int_0^{\beta_s} \sin^{n-2}\theta\,d\theta \leq \int_0^{\beta_s} \frac{\cos \theta}{\cos \beta_s} \sin^{n-2}\theta\,d\theta = \frac1{\cos \beta_s}\frac{\sin^{n-1}\beta_s}{n-1}.
\]
and
\begin{equation}\label{integral.seno}
\int_0^{\beta_s} \sin^{n-2}\theta\,d\theta \geq \int_0^{\beta_s} \cos\theta \sin^{n-2}\theta\,d\theta =\frac{\sin^{n-1}\beta_s}{n-1},
\end{equation}
We start with the upper bound for $\gamma_n(B(x_n,R_n))$. Calling $F(s^2)=\sin^2 \beta_s s^2 \e^{-s^2}$, we have
$$
 \gamma_n(B(x_n,R_n))\leq \frac{\omega_{n-2}}{n-1} T_n^n\int_{1-R}^{1+R} \frac1{\cos\beta_s} F(s^2)^\frac{n-1}2\,ds.
$$

\bigskip

By the cosine law applied to the triangle whose vertices are given by the origin, $x_n'$ and any $y\in\partial B_s \cap \partial B(x_n',R)$, one obtains
\[
\cos\beta_s=\frac{1+s^2-R^2}{2s},
\]
and consequently
\[
\sin \beta_s=\left(1-\left(\frac{1+s^2-R^2}{2s}\right)^2\right)^\frac12.
\]
The maximal value of $\beta_s$ occurs when $\partial B_s$ and $\partial B(x_n',R)$ meet perpendicularly, and then $\sin \beta_s=R$. Thus one always has $\cos\beta_s\geq \sqrt{1-R^2}$. Taking all this into account, we get
\[
 \gamma_n(B(x_n,R_n)) \leq \frac{\omega_{n-2}\,T_n^n}{(n-1)\sqrt{1-R^2}} \int_{(1-R)^2}^{(1+R)^2} F(t)^\frac{n-1}2\,\frac{dt}{2\sqrt t},
\]
where $F(t)=\left(t-\left(\frac{1+t-R^2}2\right)^2\right)\e^{-t}$.

\bigskip

It is easy to check that $F((1-R)^2)=F((1+R)^2)=0$ and that $F$ is increasing in the interval $((1-R)^2,t_0)$ and decreasing in $(t_0,(1+R)^2)$, where $t_0=2+R^2-\sqrt{1+4R^2}$ is the maximum point. So we can estimate
$$
\gamma_n(B(x_n,R_n))\leq \frac{\omega_{n-2}\,T_n^n}{(n-1)\sqrt{1-R^2}} \int_{(1-R)^2}^{(1+R)^2}F(t_0)^\frac{n-1}2\,\frac{dt}{2\sqrt t}= \frac{2\,\omega_{n-2}\,T_n^n\,R}{(n-1)\sqrt{1-R^2}} F(t_0)^\frac{n-1}2.
$$

\bigskip

Next, we obtain a lower bound for $\gamma(B_{r_n}\cap B(x_n,R_n))$. As above we have
\[
\gamma(B_{r_n}\cap B(x_n,R_n)) = \int_{T_n-R_n}^{r_n}\Bigl|(\partial B_\rho\cap B(x_n,R_n)\Bigr|_{n-1} \e^{-\pi\rho^2}\,d\rho,
\]
and by (\ref{surface.cap}) and (\ref{integral.seno})
\[
\gamma(B_{r_n}\cap B(x_n,R_n))\geq \frac{\omega_{n-2}}{n-1} T_n^n\int_{1-R}^{r} F(s^2)^\frac{n-1}2\,ds = \frac{\omega_{n-2}T_n^n}{n-1} \int_{(1-R)^2}^{r^2} F(t)^\frac{n-1}2\,\frac{dt}{2\sqrt t},
\]
As $R<1$ we have that $t_0<1$, and it will be very convenient to choose $r^2=t_0$. As $F$ is a smooth function, we can write
$F(t)=F(t_0)+\frac{F''(\tau_t)}2(t-t_0)^2$, with $\tau_t$ a point between $t$ and $t_0$. We denote by $M$ the maximum value of $|F''|$ in the interval $[(1-R)^2,(1+R)^2]$. So if $0<\delta <t_0-(1-R)^2$
\begin{eqnarray*}
\int_{(1-R)^2}^{t_0} F(t)^\frac{n-1}2\,\frac{dt}{2\sqrt t} &\geq& \int_{t_0-\delta}^{t_0} \left(F(t_0)+\frac{F''(\tau_t)}2\,(t-t_0)^2\right)^\frac{n-1}2\,\frac{dt}{2\sqrt t} \\&\geq& F(t_0)^\frac{n-1}2 \int_{t_0-\delta}^{t_0} \left(1-\frac M{2F(t_0)}\,\delta^2\right)^\frac{n-1}2\,\frac{dt}{2\sqrt t},
\end{eqnarray*}
the last inequality provided $\delta$ is small enough to make the last parenthesis positive. Choosing $\delta=\sqrt{\frac{4F(t_0)}{(n-1)M}}$, we will have $(1-\frac{M}{2F(t_0)}\delta^2)^\frac{n-1}2>c_0>0$ for $n$ large enough. Hence, the last expression is greater than or equal to
\[
c_0 \int_{t_0-\delta}^{t_0}\,\frac{dt}{2\sqrt t} F(t_0)^\frac{n-1}2\geq c_0 F(t_0)^\frac{n-1}2 \frac\delta{2\sqrt{t_0}}.
\]

\bigskip

Putting together all the estimations, we conclude
\[
\frac{\gamma_n(B(x_n,R_n)\cap B_{r_n})}{\gamma_n(B(x_n,R_n))} \geq \frac c{\sqrt{n-1}}.
\]
where $c>0$ may depend on $R$ and $r$, but not on $n$. Observe finally that $r$ is determined by $R$ via $t_0$, and that $R$ can be chosen arbitrarily in $(0,1)$.
\end{proof}

\bigskip

The proof of Theorem \ref{teorema.exponencial} follows the same scheme as the previous one, so we just hint the main steps. It is enough to show the following analogue of Lemma \ref{prueba.gaussiana}:

\bigskip

\begin{lema}\label{lema.exponencial.alfa}
There exist sequences $\{x_n\}, \{r_n\}$ and $\{R_n\}$, with $x_n\in\real^n$ and $r_n,\,R_n>0$ for $n\in\mathbb N$ such that
\begin{equation}\label{estimacion.maximal.exp.alpha}
\frac{\gamma_{\alpha,n}(B(x_n,R_n)\cap B_{r_n})}{\gamma_{\alpha,n}(B(x_n,R_n))}\geq \frac{c}{\sqrt n},
\end{equation}
and
\begin{equation}\label{cociente.bolas.centradas.exp.alpha}
\frac{\gamma_{\alpha,n}(B_{|x_n|})}{\gamma_{\alpha,n}(B_{r_n})} \geq c\,a^n,
\end{equation}
for some $a=a(\alpha)>1$.
\end{lema}

\bigskip

\begin{proof}
We first deal with the proof of (\ref{cociente.bolas.centradas.exp.alpha}). The measure of a centred ball is $\gamma_{\alpha,n}(B_\rho) = \omega_{n-1}\int_0^\rho f_{\alpha,n}(t)\,dt$, where $f_{\alpha,n}(t)=\e^{-t^\alpha}t^{n-1}$. This function attains its maximum at the radius $T_{\alpha,n}=((n-1)/\alpha)^{1/\alpha}$, around which an essential part of the mass is concentrated. For $\rho<T_{\alpha,n}$ we have as well that
\begin{equation}\label{bolas.centradas.exp.alpha}
\e^{-\rho^\alpha}|B_\rho| \leq \gamma_\alpha(B_\rho) \leq n\,\e^{-\rho^\alpha}|B_\rho|.
\end{equation}
Take $r_n=r T_{\alpha,n}$ and $x_n=T_{\alpha,n} x_n'$ with $r<1$ and $x_n'$ a unit vector. Inequalities (\ref{bolas.centradas.exp.alpha}) imply that
$$
\frac{\gamma_{\alpha,n}(B_{|x_n|})}{\gamma_{\alpha,n}(B_{r_n})} \geq \frac {\e^{-T_{\alpha,n}^\alpha}\,|B_{T_{\alpha,n}}|} {n\,\e^{-r^\alpha T_{\alpha,n}^\alpha}\ r^n\,|B_{T_{\alpha,n}}|} = \frac {\e^{(1-r^\alpha)/\alpha}}n\left(\e^{(r^\alpha-1)/\alpha}\,\frac 1r\right)^n.
$$
It is easy to see that $\e^{(r^\alpha-1)/\alpha}/r>1$ by applying the inequality $e^x> 1+x$ to $e^{r^\alpha-1}$.

\bigskip

To prove (\ref{estimacion.maximal.exp.alpha}) take $R_n =R T_{\alpha,n}$ with $1-r<R<r$. Following the steps of the proof of (\ref{estimacion.maximal.gaussiana}), we have
\begin{eqnarray*}
\frac{\omega_{n-2}\,T_{\alpha,n}^n}{n-1} \int_{(1-R)^2}^{(1+R)^2} F_\alpha(t)^\frac{n-1}2\,\frac{dt}{2\sqrt t}&\leq& \gamma_{\alpha,n}(B(x_n,R_n))\\&\leq& \frac{\omega_{n-2}T_{\alpha,n}^n}{(n-1)\sqrt{1-R^2}} \int_{(1-R)^2}^{(1+R)^2} F_\alpha(t)^\frac{n-1}2\,\frac{dt}{2\sqrt t},
\end{eqnarray*}
where $F_\alpha(t) = \left(t-\left(\frac{1+t-R^2}2\right)^2\right) \e^{-2t^{\alpha/2}/\alpha}$. This function attains its maximum at a point $t_\alpha < 1$. This is a consequence of the following facts: $F_\alpha((1-R)^2) = F_\alpha((1+R)^2)=0$, $F_\alpha(t) > 0$ for $(1-R)^2 < t < (1+R)^2$, and $F_\alpha'(t) < 0$ whenever $1 \leq t<(1+R)^2$. To see the last assertion, write the derivative of $F_\alpha$ as
\begin{eqnarray*}
\frac\partial{\partial t}F_\alpha(t)&=&\left\{1-\frac{1+t-R^2}2-t^{\alpha/2-1} \left(t-\left(\frac{1+t-R^2}2\right)^2\right)\right\} \e^{-2t^{\alpha/2}/\alpha}\\&=:& G_\alpha(t)\e^{-2t^{\alpha/2}/\alpha}.
\end{eqnarray*}
Now it is clear that for $\alpha>0$ and $1<t<1+R^2$
\[
G_\alpha(t) < G_0(t)= t^{-1}\left(\frac{1+t-R^2}2\right)^2-\frac{1+t-R^2}2<0.
\]
All this was to justify that we can take $r=\sqrt t_\alpha<1$. Now we just follow the same steps as in the proof of Lemma \ref{lema.exponencial.alfa} to estimate

$$
\frac{\gamma_{\alpha,n}(B(x_n,R_n)\cap B_{r_n})}{\gamma_{\alpha,n}(B(x_n,R_n))}\geq \frac c{\sqrt{n-1}},
$$
where the constant $c$ may depend on $r$, $R$ and $\alpha$ but not on $n$.
\end{proof}

\bigskip

\subsection*{Acknowledgements}
This work started while the first author was visiting the University of Gothenburg. He is grateful to this institution for its hospitality. The authors would also like to thank Professor Fernando Soria for his help, in particular for pointing out the connection with the Hardy operator and his contribution in the proof of Proposition \ref{cociente.bolas}.

\bigskip


\begin{thebibliography}{HD}
\normalsize
\baselineskip=17pt
\bibitem{Aldaz} J.M. Aldaz, \emph{Dimension dependency of the weak type (1,1) bounds for maximal functions associated to finite radial measures.} Bull. Lond. Math. Soc. \textbf{39} (2007), 203--208.
\bibitem{Aldaz2} J.M. Aldaz, \emph{The weak type (1,1) bounds for the maximal function associated to cubes grow to infinity with the dimension.} Ann. of Math. (2) \textbf{173} (2011), no. 2, 1013-–1023.
\bibitem{AldazPerezLazaro} J.M. Aldaz, J. P\'erez L\'azaro, \emph{Dimension dependency of $L^p$ bounds for maximal functions associated to radial measures.} Positivity \textbf{15} (2011), 199--213.
\bibitem{Aubrun} G. Aubrun, \emph{Maximal inequality for high-dimensional cubes.} Confluentes Math. \textbf{1} (2009), no. 2, 169--179.
\bibitem{Bourgain1} J. Bourgain, \emph{On high-dimensional maximal function associated to convex bodies.} Amer. J. Math. \textbf{108} (1986), no. 6, 1467--1476.
\bibitem{Bourgain2} J. Bourgain, \emph{On the $L\sp p$-bounds for maximal functions associated to convex bodies in $R\sp n$.} Israel J. Math. \textbf{54} (1986), no. 3, 257--265.
\bibitem{Bourgain3} J. Bourgain, \emph{On dimension free maximal inequalities for convex symmetric bodies in $\real^n$.} Geometrical aspects of functional analysis (1985/86), 168--176, Lecture Notes in Math., \textbf{1267}, Springer, Berlin, 1987.
\bibitem{Carbery} A. Carbery, \emph{An almost-orthogonality principle with applications to maximal functions associated to convex bodies.} Bull. Amer. Math. Soc. (N.S.) \textbf{14} (1986), no. 2, 269--273.
\bibitem{C} A. Criado, \emph{On the lack of dimension free estimates in $L^p$ for maximal functions associated to radial measures.} Proc. Roy. Soc. Edinburgh Sect. A \textbf{140} (2010), no. 3, 541–-552.
\bibitem{Guzman} M. de Guzm\'an, \emph{Differentiation of Integrals in $\real^n$.} Lecture Notes in Math., Vol. 481. Springer-Verlag, Berlin-New York, 1975.
\bibitem{AdrianInfante} A. Infante, \emph{An\'alisis arm\'onico para medidas no doblantes: operadores maximales con respecto a la medida gaussiana}, Ph.D. thesis, Universidad Aut\'onoma de Madrid, 2005.
\bibitem{MenarguezSoria} M.T. Men\'arguez, F. Soria, \emph{Weak type (1,1) inequalities for maximal convolution operators.} Rend. Circ. Mat. Palermo (2) \textbf{41} (1992), no. 3, 342--352.
\bibitem{MenarguezSoria2} M.T. Men\'arguez, F. Soria, \emph{On the maximal operator associated to a convex body in $\real^n$.} Collect. Math. \textbf{3} (1992), 243--251.
\bibitem{Muller} D. M\"uller, \emph{A geometric bound for maximal functions associated to convex bodies.} Pacific J. Math. \textbf{142} (1990), no. 2, 297--312.
\bibitem{NaorTao} A. Naor, T. Tao \emph{Random martingales and localization of maximal inequalities}, J. Funct. Anal. \textbf{259} (2010), no. 3, 731–-779,
\bibitem{Stein1} E.M. Stein, \emph{The development of square functions in the work of A. Zygmund.} Bull. Amer. Math. Soc. (N.S.) \textbf{7} (1982), no. 2, 359--376.
\bibitem{SteinStromberg} E.M. Stein, J.O. Str\"omberg, \emph{Behavior of maximal functions in $R^n$ for large n.} Ark. Mat. \textbf{21} (1983), no. 2, 250--269.
\bibitem{SteinWeiss} E.M. Stein, G. Weiss, \emph{Introduction to Fourier analysis on Euclidean spaces.} Princ. Math. Ser., No. 32. Princeton University Press, Princeton, N.J., 1971.
\end{thebibliography}
\end{document}